\documentclass[12pt]{article}
\usepackage{amsmath, amssymb}

\usepackage{geometry}
\usepackage{setspace}
\usepackage{hyperref}
\usepackage{graphicx}
\usepackage{subcaption}
\usepackage{amsthm}
\newtheorem{theorem}{Theorem}

\usepackage{amsthm} % Added for proof

\begin{document}

\title{ Analysis of Discrete Stochastic Population Models with Normal Distribution}
\author{Haiyan Wang\\
School of Mathematical and Natural Sciences\\
Arizona State University\\
Phoenix, AZ 85069\\
haiyan.wang@asu.edu
}
\date{}
\maketitle

\begin{abstract}
This paper analyzes a stochastic logistic difference equation under the assumption that the population distribution follows a normal distribution. Our focus is on the mathematical relationship between the average growth rate and a newly introduced concept—the uniform structural growth rate, which captures how growth is influenced by the internal distributional structure of the population. We derive explicit relationships linking the uniform structural growth rate to the parameters of the normal distribution and the variance of a small stochastic perturbation. The analysis reveals the existence of two distinct branches of the uniform structural growth rate, corresponding to alternative population states characterized by higher and lower growth rates. This duality provides deeper insights into the dynamics of population growth under stochastic influences. A sufficient condition for the existence of two uniform structural growth rates is established and rigorously proved, demonstrating that there exist infeasible intervals where no uniform structural growth rate can be defined. We also explore the biological significance of these findings, emphasizing the role of stochastic perturbations and the distribution in shaping population dynamics.

\end{abstract}

\section{Introduction}
Logistic difference equations are commonly used to model population growth in environments with limited resources. In their deterministic form, these models capture an initial phase of exponential growth that gradually slows as the population nears the carrying capacity—the maximum population size that the environment can support. The behavior of these models is highly sensitive to the intrinsic growth rate, which governs the rate of expansion and the dynamics of the system \cite{Marry2002, may1976simple,kot2001elements}.  

Stochastic models provide important advantages over their deterministic counterparts. While deterministic models yield fixed outcomes based solely on initial conditions, stochastic models capture the full distribution of possible population sizes, offering deeper insights into both the expected dynamics and the variability around them. By incorporating randomness, these models more accurately reflect the behavior of real-world populations, especially under fluctuating environmental conditions. The study of stochastic equations has become a vibrant area of research, with extensive literature exploring a wide range of extensions and applications \cite{Allen,Aktar2023,kot2001elements,Schreiber2021,Braverman2013,Yan2024,erguler2008statistical,Kang2010}.

In this paper,  we study the stochastic logistic equation from time $t$ to $t+1$
\begin{equation} \label{eq:logistic3}
X_{t+1} = r X_t \left( 1 - X_t \right) \epsilon_t,
\end{equation}
where \( X_t \) is a distribution of population size at time \( t \), \( r >0\) is a real number, $\epsilon_t$ is a small nonnegative perturbation distribution representing stochastic effects.  $\epsilon_t$ is independent of $X_t$ and its mean is $1$ ($E[\epsilon_t] = 1$).  We assume that $X_t$ follows a normal distribution ($X_t \sim N(\mu, \sigma^2), \mu>0, \sigma>0)$.

Here we call \( r \) a \textit{uniform intrinsic growth rate under structural transformation}, abbreviated as the \textit{uniform structural growth rate} — a constant applied uniformly across the entire of the domain of $X_t$, following a nonlinear structural transformation given by $X_t(1-X_t)$. This differs from an average growth rate defined by the expectation $E[rX_t(1-X_t)]$, which captures overall outcomes but not underlying structure. We try to address this question: \textit{In order to achieve an average growth rate $\alpha>0$ $(E[X_{t+1}] = \alpha E[X_{t}]$), what are the required uniform structural growth rate $r$?}  

Our analysis in this paper indicates that there are multiple branches of uniform structural growth rates, which represent alternative states corresponding to higher and lower growth rates, respectively. We establish the mathematical relationship between the uniform structural growth rate \( r \) and the parameters of the normal distribution. A sufficient condition is presented to guarantee the existence of two uniform structural growth rates. We demonstrate that there are unfeasible intervals of $E[X_{t}]$ with which the population does not have a uniform structural growth rate to achieve $(E[X_{t+1}] = \alpha E[X_{t}]$. This duality provides deeper insights into population growth under stochastic conditions. We further investigate the impact of the parameters of the normal distribution and describe biological interpretations. The influence of the parameters of the normal distribution and perturbation \( \epsilon_t \) on \( r \) provides theoretical insights into how the expectation and variance of the underlying population impact uniform structural growth rates. 

The work extends the analysis in \cite{Wang2025} where the gamma distribution is used to address the growth rate of the stochastic logistic and Ricker logistic equations at equilibrium. In this paper, the population is assumed to be in a state following a normal distribution. The study offers deeper insights into the interplay among stochastic effects, uniform structural growth rates, nonlinear transformation and distributional parameters, providing a more comprehensive framework for understanding real-world population dynamics.

\section{Mathematical formulation}

Let's state our major result on the relation of $r$ with the parameters of the normal distribution as a theorem. 
\begin{theorem}
For stochastic logistic equation \eqref{eq:logistic3}, assume that $X_t$ follows a normal distribution $N(\mu, \sigma^2)$,  and $E[\epsilon_t] = 1, E[\epsilon_t^2] = v$ with
parameters \( v > 0 \), $\sigma^2>0$, \( 0 < \mu < 1 \). In addition, we assume that 
\begin{align}
E[X_{t+1}] &= \alpha E[X_{t}], \\
\operatorname{Var}[X_{t+1}] &= \beta \operatorname{Var}[X_{t}],
\label{assmp1}
\end{align}
where \(\beta > 0 \), \(\alpha > 0 \). Then $r$ in \eqref{eq:logistic3} are positive solutions of the following polynomial in terms of $r$: 
\begin{align}
\begin{aligned}
&v (-2 \mu^4 + 4 \mu^3 - 3 \mu^2 + \mu) r^3 - v \alpha \mu r^2 \\
&\quad + \left[ v (3 \alpha^2 \mu^2) + \beta \mu (\mu - 1) - \alpha^2 \mu^2 \right] r + \beta \alpha \mu = 0
\end{aligned} 
\label{r-relation}
\end{align}
\label{thm1}
\end{theorem}

\begin{proof}
\textbf{Moments of $X$}: We have the following relations from Appendix \ref{appendix}
\begin{align}
\begin{aligned}
E[X_t] &= \mu, & E[X_t^2] &= \mu^2 + \sigma^2, \\
E[X_t^3] &= \mu^3 + 3\mu \sigma^2, & E[X_t^4] &= \mu^4 + 6\mu^2 \sigma^2 + 3\sigma^4,
\end{aligned}
\label{r1}
\end{align}

\textbf{Expectation condition}: We begin with the expectation condition. Because of $E[\epsilon_t] = 1,$ we have 
\begin{align}
\begin{aligned}
E[X_{t+1}] &= r E[X_t (1 - X_t)], \\
&= r \left( \mu - (\mu^2 + \sigma^2) \right), \\
\end{aligned}
\end{align}
With \eqref{assmp1}, it follows that

\begin{align}
\begin{aligned}
&r \left( \mu - \mu^2 - \sigma^2 \right) = \alpha \mu, \\
& \sigma^2 = \mu - \mu^2 - \frac{\alpha \mu}{r}.
\label{exp1}
\end{aligned}
\end{align}

\textbf{Variance condition}: We continue with variance condition in \eqref{assmp1}
\begin{align}
\operatorname{Var}[X_{t+1}] &= E[X_{t+1}^2] - (E[X_{t+1}])^2=\beta \sigma^2
\label{eq:19}
\end{align}
and 
\begin{align}
E[X_{t+1}^2] &= \beta \sigma^2 + (\alpha \mu)^2, \label{eq:17}
\end{align}
On the other hand, from \eqref{eq:logistic3}, we have
\begin{align}
\begin{aligned}
E[X_{t+1}^2] &= r^2 E[\epsilon_t^2] E\left[ X_t^2 (1 - X_t)^2 \right], \\
&= r^2 v E\left[ X_t^2 - 2 X_t^3 + X_t^4 \right]
\end{aligned}
\end{align}
In view of \eqref{r1}
\begin{align}
\begin{aligned}
E[X_{t+1}^2] &= r^2 v \left[ (\mu^2 + \sigma^2) - 2 (\mu^3 + 3\mu \sigma^2) \right. \\
&\quad \left. + (\mu^4 + 6\mu^2 \sigma^2 + 3\sigma^4) \right], \\
&= r^2 v \left[ \mu^4 - 2\mu^3 + \mu^2 + (6\mu^2 \sigma^2 - 6\mu \sigma^2 + \sigma^2)+ 3\sigma^4 \right].
\label{eq:17b}
\end{aligned}
\end{align}
Thus we arrive at this relation with all the parameters. 

\begin{align}
r^2 v \left[ \mu^4 - 2\mu^3 + \mu^2 + (6\mu^2 \sigma^2 - 6\mu \sigma^2 + \sigma^2) + 3\sigma^4 \right] &= \beta \sigma^2 + (\alpha \mu)^2, 
\label{eq:18}
\end{align}
We will eliminate $\sigma$ using \eqref{exp1}. Note that 
\begin{align}
\begin{aligned}
\sigma^4 &= \left( \mu - \mu^2 - \frac{\alpha \mu}{r} \right)^2, \\
&= \left( \mu - \mu^2 \right)^2 - 2 \left( \mu - \mu^2 \right) \frac{\alpha \mu}{r} + \frac{\alpha^2 \mu^2}{r^2}, \\
&= \mu^2 - 2\mu^3 + \mu^4 - \frac{2\alpha \mu^2}{r} + \frac{2\alpha \mu^3}{r} + \frac{\alpha^2 \mu^2}{r^2}, \\
\end{aligned}
\end{align}
Now substitute $\sigma^2$ and $\sigma^4 $ into the left side of \eqref{eq:18}:   
\begin{align}
\begin{aligned}
&v r^2 \left[ \mu^4 - 2\mu^3 + \mu^2 + (6\mu^2 - 6\mu + 1) \sigma^2 + 3\sigma^4 \right], \\
&=v r^2 \left[ \mu^4 - 2\mu^3 + \mu^2 \right. \\
&\quad + (6\mu^2 - 6\mu + 1) \left( \mu - \mu^2 - \frac{\alpha \mu}{r} \right) \\
&\quad \left. + 3 \left( \mu^4 - 2\mu^3 + \mu^2 - \frac{2\alpha \mu^2}{r} + \frac{2\alpha \mu^3}{r} + \frac{\alpha^2 \mu^2}{r^2} \right) \right], \\
\label{mil12}
\end{aligned}
\end{align}
Expanding the terms in the left side of \eqref{mil12}, the left side of \eqref{eq:18} becomes
\begin{align}
\begin{aligned}
& \frac{1}{r} v r^3 \left[ \mu^4 - 2\mu^3 + \mu^2 \right. \\
&\quad + \left( 6\mu^3 - 6\mu^4 - \frac{6\alpha \mu^3}{r} - 6\mu^2 + 6\mu^3 + \frac{6\alpha \mu^2}{r} + \mu - \mu^2 - \frac{\alpha \mu}{r} \right) \\
&\quad + \left( 3\mu^4 - 6\mu^3 + 3\mu^2 - \frac{6\alpha \mu^2}{r} + \frac{6\alpha \mu^3}{r} + \frac{3\alpha^2 \mu^2}{r^2} \right) \bigg], \\
&\quad =\frac{1}{r} v r^3 \left[ -2\mu^4 + 4\mu^3 - 3\mu^2 + \mu  - \alpha \mu \frac{1}{r} + 3\alpha^2 \mu^2 \frac{1}{r^2} \right],\\
&\quad =\frac{1}{r} v  \left[ -2\mu^4 r^3 + 4\mu^3 r^3 - 3\mu^2 r^3 + \mu r^3 - \alpha \mu r^2 + 3\alpha^2 \mu^2 r \right].
\label{left1}
\end{aligned}
\end{align}
On the other hand, the right side of \eqref{eq:18} is

\begin{align}
\begin{aligned}
& \beta \sigma^2 + (\alpha \mu)^2, \\
&=  \beta \left( \mu - \mu^2 - \frac{\alpha \mu}{r} \right) + \alpha^2 \mu^2 , \\
&=  \beta \mu - \beta \mu^2 - \frac{\beta \alpha \mu}{r} + \alpha^2 \mu^2, \\
&= \frac{1}{r} \left[\beta \mu r - \beta \mu^2 r - \beta \alpha \mu + \alpha^2 \mu^2 r \right],\\
\label{rigt1}
\end{aligned}
\end{align}
Combining \eqref{eq:18},  \eqref{left1} and \eqref{rigt1} we arrive at the following simplified equation without $\sigma$ after cancelling $\frac{1}{r}$ at the both sides.

\begin{align}
\begin{aligned}
v &\left[ -2\mu^4 r^3 + 4\mu^3 r^3 - 3\mu^2 r^3 + \mu r^3 - \alpha \mu r^2 + 3\alpha^2 \mu^2 r \right] \\
&= \beta \mu r - \beta \mu^2 r - \beta \alpha \mu + \alpha^2 \mu^2 r,
\end{aligned}
\end{align}
Combining the like terms for $r^i, i=3,2,1,0$, we conclude the proof of Theorem. 
\begin{align}
\begin{aligned}
&v (-2 \mu^4 + 4 \mu^3 - 3 \mu^2 + \mu) r^3 - v \alpha \mu r^2 \\
&\quad + \left[ v (3 \alpha^2 \mu^2) + \beta \mu (\mu - 1) - \alpha^2 \mu^2 \right] r + \beta \alpha \mu = 0
\end{aligned}
\end{align}
\end{proof}

\section{Conditions for two positive solutions}\label{mathpro}
Consider the cubic polynomial from the stochastic logistic model:
\begin{align}
p(r) = a r^3 + b r^2 + c r + d = 0,
\label{poly1}
\end{align}
where:
\begin{align}
a &= v (-2 \mu^4 + 4 \mu^3 - 3 \mu^2 + \mu), \\
b &= -v \alpha \mu, \\
c &= v (3 \alpha^2 \mu^2) + \beta \mu (\mu - 1) - \alpha^2 \mu^2, \\
d &= \beta \alpha \mu,
\end{align}
with parameters \( v > 0 \), \(\beta > 0 \), \(\alpha > 0 \), \( 0 < \mu < 1 \), and \( r \) as the growth rate. We seek conditions ensuring exactly two positive real roots and one negative root.

\begin{theorem}
Assume that parameters \( v > 0 \), \(\beta > 0 \), \(\alpha > 0 \), \( 0 < \mu < 1 \). Then \eqref{poly1} has exactly two positive solutions if 
\begin{align}
\Delta >0 \quad \text{and} \quad p\left( \frac{2 v \alpha \mu + \sqrt{\Delta}}{6 a} \right) < 0.
\end{align}
where $\Delta = 4 v^2 \alpha^2 \mu^2 - 12 a c$.
\label{thm2}
\end{theorem}

\begin{proof}
\textbf{Claim \( a>0 \)}: Let  $g(\mu)=(-2 \mu^3 + 4 \mu^2 - 3 \mu + 1)$ and 
\begin{align}
a = v \mu (-2 \mu^3 + 4 \mu^2 - 3 \mu + 1) = v \mu g(\mu).
\end{align}
Note that $g(0)=1,g(1)=0$ and $g'(\mu)= -6 \mu^2 + 8 \mu - 3 $. The discriminant of $g'(\mu)$  is $64 - 72 = -8 < 0 $ and has no real roots.  Note that \( g'(0.5) = -1.5 + 4 - 3 = -0.5 < 0 \), so \( g'(\mu) < 0 \). Thus, \( g(\mu) \) is strictly decreasing. Thus, \( g(\mu) \) is strictly decreasing. Since \( g(\mu) \) is continuous, strictly decreasing, starts at 1, and ends at 0, \( g(\mu) > 0 \) for \( \mu \in (0, 1) \). It follows that $a>0$ for $\mu \in (0,1)$

\textbf{Behavior of $p(r)$}: With \( a > 0 \), \( b < 0 \), \( d > 0 \), it is easy to see that
\begin{itemize}
    \item \( r \to -\infty \): \( p(r) \to -\infty \),
    \item \( r = 0 \): \( p(0) = d > 0 \),
    \item \( r \to +\infty \): \( p(r) \to +\infty \).
\end{itemize}
As a result, there is one root \( r_1 < 0 \), two possible roots \( r_2, r_3 > 0 \),

\textbf{Critical Points of $p(r)$:} It is easy to see that $p'(r) = 3 a r^2 + 2 b r + c$ and its discriminant:
\begin{align}
\Delta = (2 b)^2 - 4 (3 a) c = 4 v^2 \alpha^2 \mu^2 - 12 a [v (3 \alpha^2 \mu^2) + \beta \mu (\mu - 1) - \alpha^2 \mu^2].
\end{align}
The critical points of $p(r)$ is 
\begin{align}
r_{\pm} = \frac{-2 b \pm \sqrt{\Delta}}{6 a} = \frac{2 v \alpha \mu \pm \sqrt{\Delta}}{6 a}.
\end{align}
In view of the assumption of Theorem \ref{thm2}, there are two critical points of $p'(r)$ and \( r_- < r_+ \) and \( r_+ > 0 \).  

\textbf{Second Derivative of $p(r)$:}  $p''(r) = 6 a r + 2 b$  and the inflection point of $p(r)$ is $r_i = \frac{v \alpha \mu}{3 a} > 0$. Therefore, 
\begin{itemize}
    \item \( r_- < r_i \): \( p''(r_-) < 0 \) (concave down),
    \item \( r_+ > r_i \): \( p''(r_+) > 0 \) (concave up).
    \item Local maximum at \( r_- \), \( p(r_-) > 0 \),
    \item Local minimum at \( r_+ > 0 \), \( p(r_+) < 0 \) (by the assmption).
\end{itemize}
Because of $p(0)>0$, there are two positive roots $r_2$ and $r_3$  of $p(r)$ such that  $r_1< 0<r_2< r_3$.   We can see that these points follow this order:  \( r_1<r_{-}<r_i< r_2 < r_+ <r_3\). This completes the proof.    
\end{proof}

 \begin{figure}[h!]
    \centering
    \begin{subfigure}[b]{0.80\textwidth}
        \centering
        \includegraphics[width=\textwidth]{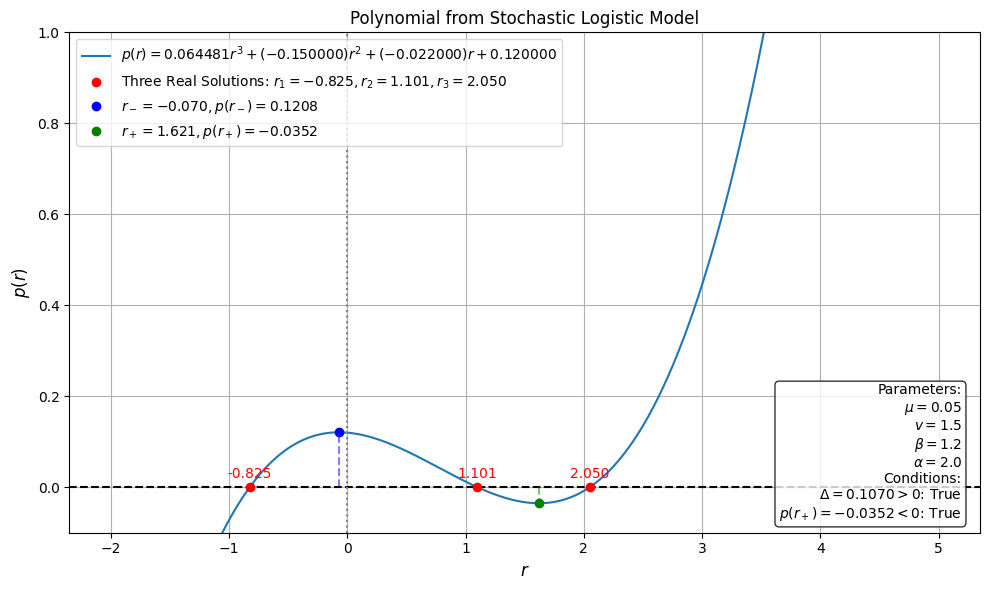}
        \label{fig:sub1}
    \end{subfigure}
    \caption{Plots of $p(r)$ with two positive real solutions }
    \label{fig:main10}
\end{figure}

 \begin{figure}[h!]
    \centering
    \begin{subfigure}[b]{0.80\textwidth}
        \centering
        \includegraphics[width=\textwidth]{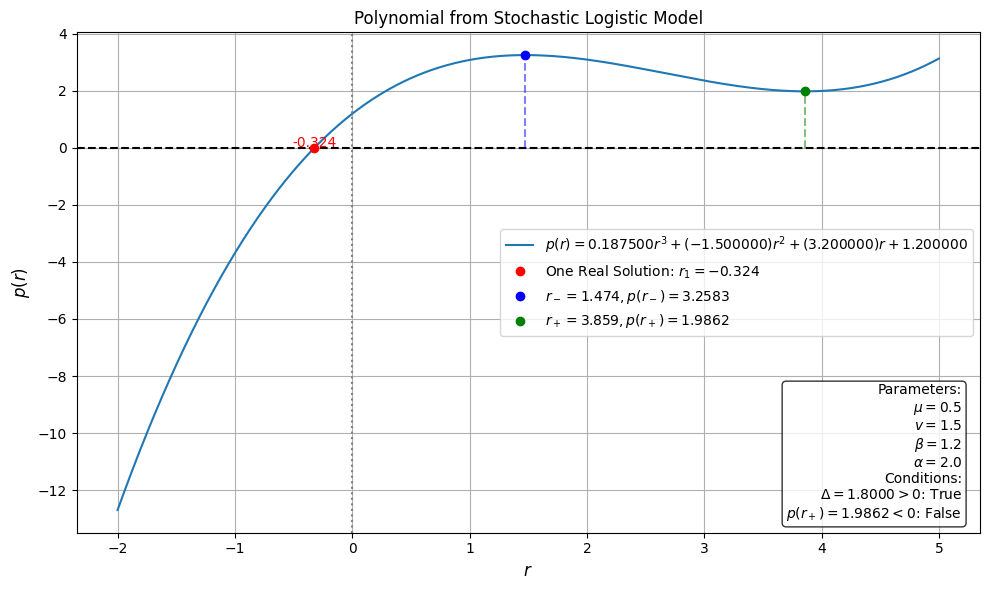}
%        \caption{Caption 2}
        \label{fig:sub20}
    \end{subfigure}    
    \caption{Plots of $p(r)$ with no positive real solution}
    \label{fig:main11}
\end{figure}
Figure \ref{fig:main10} depicts the graph of $p(r)$ for the corresponding parameters ( $\mu = 0.05, v = 1.5, \beta = 1.2, \alpha = 2.0$), critical points $r_-, r_+$, along with relevant information for its roots in the proof. As we can see, the two conditions hold. It has one negative real solution ($r_1$) and two positive real solutions ($r_2, r_3$). On the other hand, in Figure \ref{fig:main11}, the parameters are $\mu = 0.5, v = 1.5, \beta = 1.2, \alpha = 2.0$, the first condition holds, but the second condition does not hold. It have only one negative root $r_1$. 

\section{Biological Interpretation}

In this section, we plot the graph of $r$ in terms of $\mu$ to understand the relation of $r$ with other parameters. We would like to see the impact of the parameters $\alpha, \beta, v$.  Figures \ref{fig:main100}, \ref{fig:main101}, \ref{fig:main102} illustrate that there are two solutions for \( r \) in \eqref{r-relation}, representing alternative uniform structural growth rate \( r \). The upper branch typically corresponds to higher growth rates, indicating a more resilient and robust population state. This branch is associated with higher population densities and is likely to represent a stable state for species that thrive in resource-rich environments. In contrast, the lower branch reflects lower growth rates and corresponds to populations that are more vulnerable to stochastic fluctuations. These lower-density populations are generally more sensitive to environmental variability and may face a heightened risk of extinction, particularly under conditions of high environmental variance.

There is a subinterval of $\mu$ in $(0,1)$ with no unform structural growth rate $r$, which represents the unfeasible range of \( \mu \). This is consistent with Theorem \ref{thm2} which has the two conditions to guarantee the existence of two positive solutions. The unfeasible range of $\mu$ is constrained by the nonlinear term of the logistic equation \eqref{eq:logistic3} and its peak occurs in the middle of the interval $(0,1)$, which makes harder to achieve $E[X_{t+1}] = \alpha E[X_{t}]$.

The figures show that the right branch generally lies above the left, indicating that the uniform structural growth rate $r$ tends to be larger for feasible values of $\mu$ near $1$. This behavior is partly due to the nonlinear transformation $X(1-X)$. Although the graph $X(1-X)$ is symmetric on $(0,1)$, a higher uniform structural growth rate is required to achieve  $E[X_{t+1}] = \alpha E[X_{t}]$ when $E[X_{t}]$ is close to $1$. 

From a population perspective, the fact that the right branch generally lies above the left implies that higher population densities (i.e., values of $\mu$ closer to 1) are associated with larger uniform structural growth rates. This suggests that populations in high-density states require greater growth potential to maintain or increase their size under stochastic influences. Biologically, this reflects the diminishing marginal returns of the nonlinear term $X(1-X)$ as populations approach carrying capacity. Even though the resource limitations intensify near high density, the system compensates with a higher growth rate to sustain the desired average population level. Conversely, populations at lower densities (left branch) can maintain growth with smaller values of $r$, reflecting fewer intraspecific constraints and greater sensitivity to environmental changes.

\subsection{Impact of $\alpha$}

Figures \ref{fig:main100} illustrates the impact of $\alpha$. Here we compare the two cases with the parameters $\alpha = 1.1,  E[\epsilon_t^2], = 1.2, \beta = 1.1$ and  $\alpha = 1.4,  E[\epsilon_t^2], = 1.2, \beta = 1.1$. As $\alpha $ becomes larger, $r$ becomes larger to achieve $E[X_{t+1}] = \alpha E[X_{t}]$. 

At the population level, an increase in $\alpha$ reflects a greater expected growth multiplier from one generation to the next. To achieve this higher expected growth under stochastic dynamics, the population must compensate with a higher uniform structural growth rate $r$. This indicates that, in more rapidly expanding populations, stronger intrinsic reproductive potential is required to offset the variability introduced by environmental noise and internal structural constraints. In ecological terms, it suggests that populations aiming for faster growth must possess higher resilience or reproductive capacity to sustain that trajectory in uncertain environments.

 \begin{figure}[h!]
    \centering
    \begin{subfigure}[b]{0.45\textwidth}
        \centering
        \includegraphics[width=\textwidth]{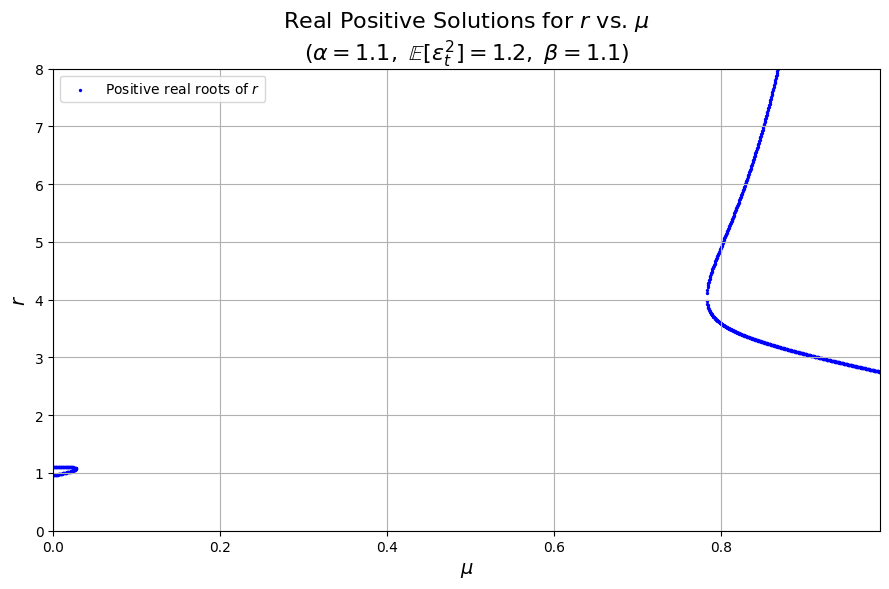}
        \label{fig:sub1}
    \end{subfigure}
    \hfill
    \begin{subfigure}[b]{0.45\textwidth}
        \centering
        \includegraphics[width=\textwidth]{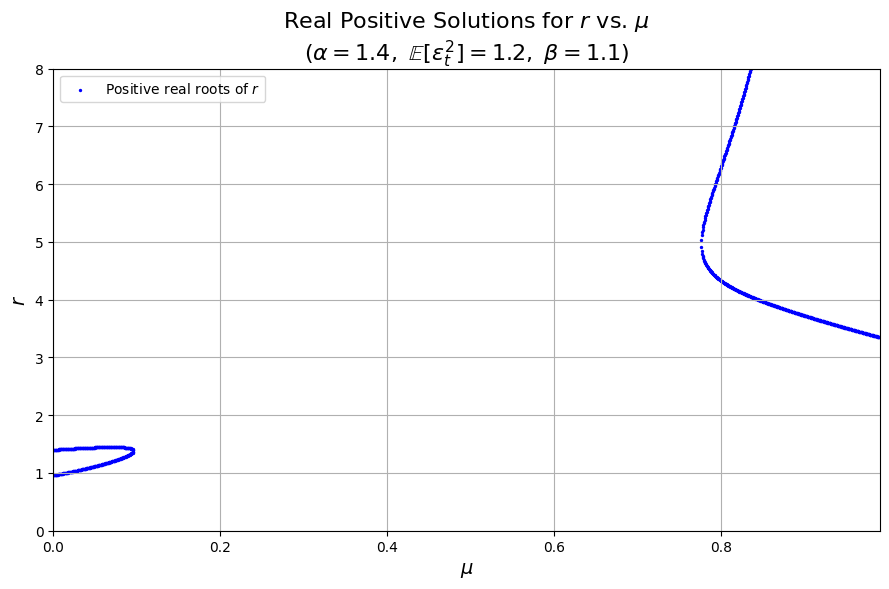}
%        \caption{Caption 2}
        \label{fig:sub2}
    \end{subfigure}
    \caption{Plots of $r$ in terms of $\mu$. Impact of $\alpha$}
    \label{fig:main100}
\end{figure}

\subsection{Impact of $\beta$}

Figure \ref{fig:main101} illustrates the impact of $\beta$ when other parameters are fixed. Here we compare the two cases with the parameters $\alpha = 0.5,  E[\epsilon_t^2], = 1.2, \beta = 1.1$ and  $\alpha = 0.5,  E[\epsilon_t^2], = 1.2, \beta = 1.9$.  As $\beta$ increases, the uniform structural growth rate $r$ needs to be larger to achieve $E[X_{t+1}] = \alpha E[X_{t}]$ and $\operatorname{Var}[X_{t+1}] = \beta \operatorname{Var}[X_{t}]$ as $\beta$ becomes larger. This is more visible for the left branch where $\mu$ is smaller.  

From a population dynamics perspective, the parameter \( \beta \) reflects how the variance of population size evolves over time. A higher \( \beta \) indicates increasing variability from one generation to the next, which may arise from environmental disturbances, demographic fluctuations, or other sources of ecological uncertainty. To maintain the expected growth condition \( \mathbb{E}[X_{t+1}] = \alpha \mathbb{E}[X_t] \) under greater variability, a larger uniform structural growth rate \( r \) is required. This effect is particularly pronounced on the left branch of the solution curve, where the mean population size \( \mu \) is relatively small. In these low-density populations, higher variance poses a greater risk of decline or extinction, so stronger reproductive potential (i.e., a higher value of \( r \)) is needed to compensate. This highlights the vulnerability of small populations to stochastic effects and the importance of growth resilience in uncertain environments.

 \begin{figure}[h!]
    \centering
    \begin{subfigure}[b]{0.45\textwidth}
        \centering
        \includegraphics[width=\textwidth]{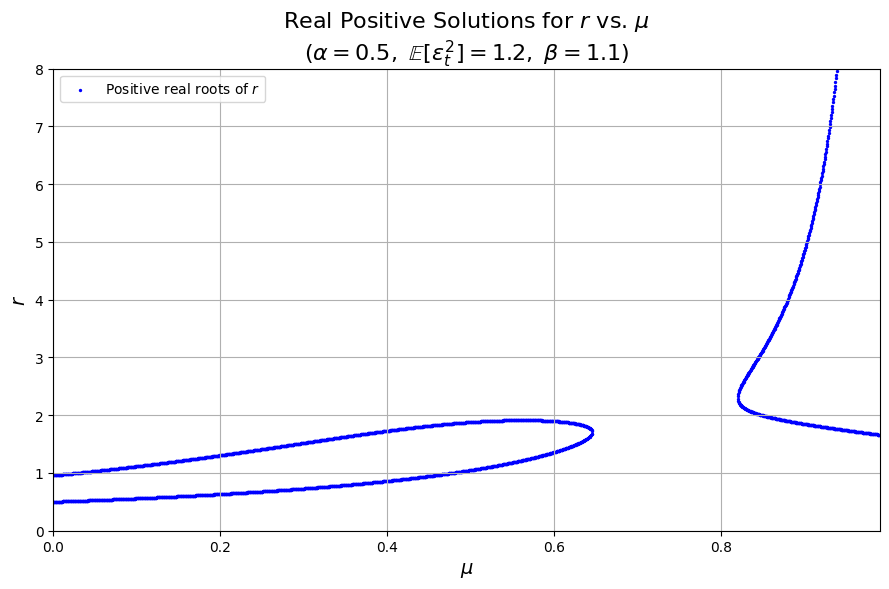}
        \label{fig:sub1}
    \end{subfigure}
    \hfill
    \begin{subfigure}[b]{0.45\textwidth}
        \centering
        \includegraphics[width=\textwidth]{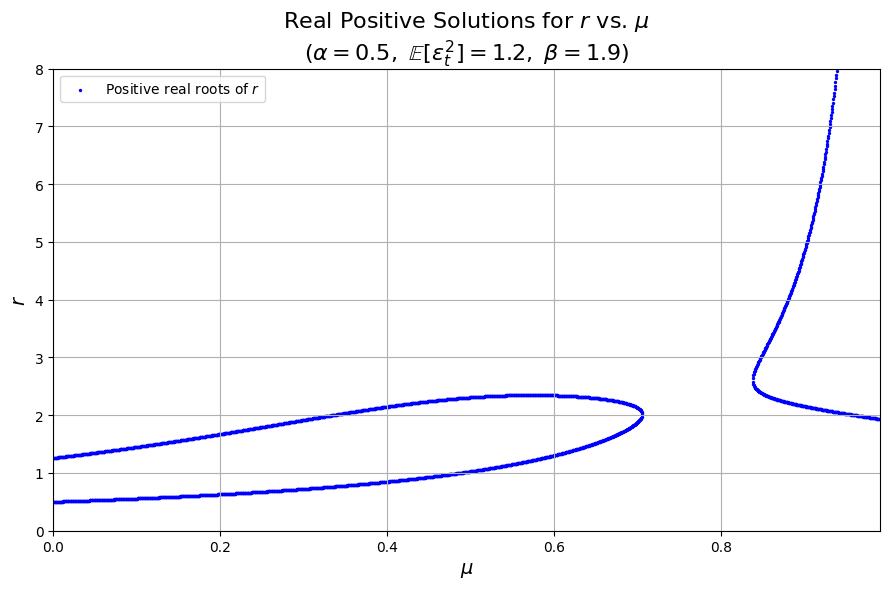}
%        \caption{Caption 2}
        \label{fig:sub2}
    \end{subfigure}
    \caption{Plots of $r$ in terms of $\mu$. Impact of $\beta$}
    \label{fig:main101}
\end{figure}

\subsection{Impact of $var[\epsilon_t]$}

Figure \ref{fig:main102} illustrates the impact of the variance \( \text{Var}[\epsilon_t] \) of the stochastic perturbation \( \epsilon_t \) in the logistic model \eqref{eq:logistic3}. We compare two cases: \( \alpha = 0.5, \mathbb{E}[\epsilon_t^2] = 1.2, \beta = 1 \) and \( \alpha = 0.5, \mathbb{E}[\epsilon_t^2] = 2.2, \beta = 1 \), where \( \mathbb{E}[\epsilon_t^2] = \text{Var}[\epsilon_t] + 1 \). As shown in the figure, increasing \( \text{Var}[\epsilon_t] \) stretches the unfeasible range of \( \mu \), meaning that a broader range of population means cannot be sustained by any uniform structural growth rate \( r \). 

In terms of population-level behavior, higher environmental variance reduces the parameter space where stable stochastic growth is possible. Additionally, the structural growth rate \( r \) becomes slightly larger on the lower branch and smaller on the upper branch with increasing \( \text{Var}[\epsilon_t] \). This suggests that in low-growth or vulnerable population states (lower branch), increased environmental variability demands a higher intrinsic growth potential to maintain the expected dynamics. In contrast, in high-growth or resilient states (upper branch), the increased stochasticity allows the system to achieve the same expected growth with a slightly reduced value of \( r \). This subtle asymmetry highlights how different population states respond differently to environmental noise, with important implications for stability and long-term viability.

 \begin{figure}[h!]
    \centering
    \begin{subfigure}[b]{0.45\textwidth}
        \centering
        \includegraphics[width=\textwidth]{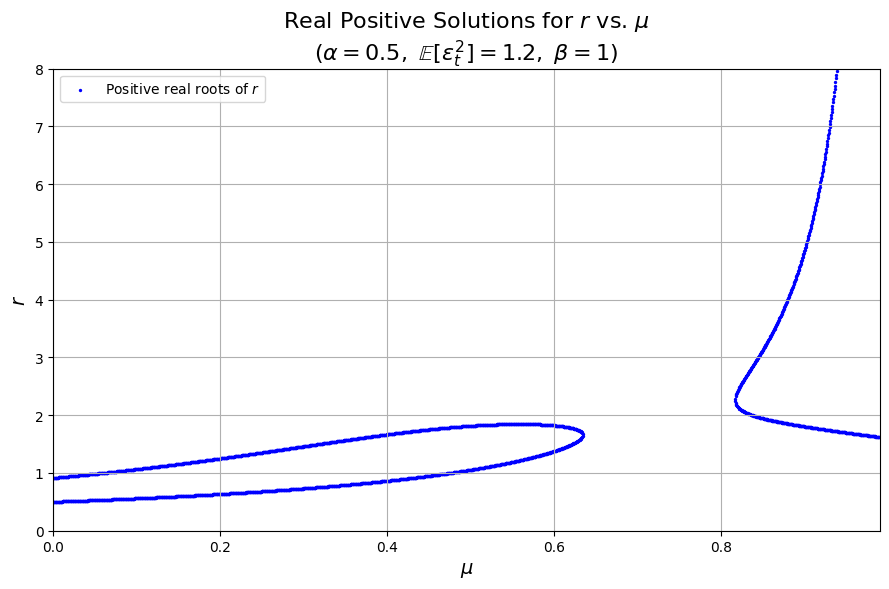}
        \label{fig:sub1}
    \end{subfigure}
    \hfill
    \begin{subfigure}[b]{0.45\textwidth}
        \centering
        \includegraphics[width=\textwidth]{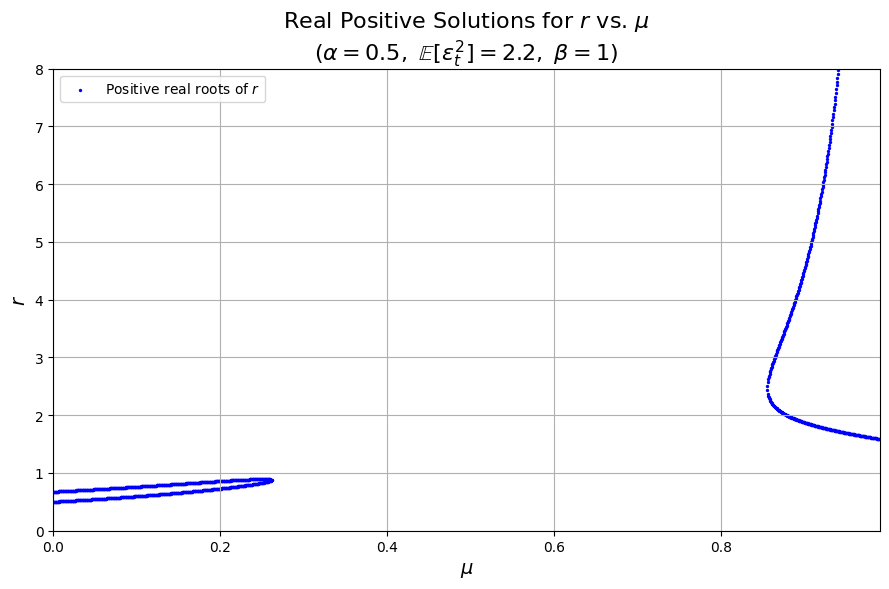}
%        \caption{Caption 2}
        \label{fig:sub2}
    \end{subfigure}
    \caption{Plots of $r$ in terms of $\mu$. Impact of $E[\epsilon_t^2]=Var[\epsilon_t]+1$}
    \label{fig:main102}
\end{figure}

\section{Discussion}
In this work, we explored the use of the normal distribution to study the populations governed by discrete a stochastic logistic equation, and established the mathematical relationships between the parameters of the normal distribution and those of the stochastic logistic model. Specifically, we formulated the uniform structural growth rate $r$ as the positive solutions of a polynomial of the parameters of the normal distribution and the variance associated with small stochastic perturbations. Based on the detailed study of the polynomial, we discovered that the uniform structural growth rate $r$ can exhibit two distinct branches, corresponding to alternative population states. This duality provides deeper insights into population growth under stochastic conditions. A sufficient condition for the existence of two uniform structural growth rates was established and rigorously proved. We also identified both feasible and infeasible ranges for the mean $\mu$ within the context of the stochastic model.

We discussed the biological significance of these findings, emphasizing the influence of stochastic perturbations and the parameter of the normal distribution on the existence of the uniform structural growth rate $r$. Ecological interpretations of these relationships were provided, highlighting their implications for understanding real-world population dynamics.

Theorems \ref{thm1} and \ref{thm2} reveal that unlike deterministic discrete models, not every $r>0$ is admissible as a uniform structural growth rate in the stochastic setting. Only two distinct branches of the uniform structural growth rate exist for equation \eqref{eq:logistic3}, with their feasibility determined by the values of \( \mu \), \( \text{Var}[\epsilon_t] \) and $\beta$. Moreover, there are regions in the parameter space where no uniform structural growth rate is possible—specifically, certain ranges of $\mu$ are infeasible for the stochastic model.

Several promising directions for future research can tackle the challenges identified in this study and build upon its findings. One immediate direction involves extending the analysis to more complex discrete stochastic models that have additional biological factors such as migration and environmental variability, along with appropriately chosen probability distributions. Such extensions may necessitate the development of new mathematical results to capture the increased complexity of real-world population dynamics. These theoretical advancements will offer deeper insights into the behavior of populations under stochastic influences and help understand how the underlying distributions affect the papulation dynamics.

Finally, real-world biological population data could be incorporated to validate and improve the mathematical relationships developed in this study. Using empirical data, appropriate probability distributions can be identified, stochastic models can be fitted, and key model parameters can be estimated. These empirical results can then be compared with the theoretical findings presented in this paper. Such an approach would allow for a rigorous assessment of how well the theoretical framework captures the behavior of fluctuating populations, thereby strengthening the model’s ecological relevance and practical applicability to real-world population dynamics.

\section{Appendix: Moments formula}\label{appendix}
The computation of the moments for a normal distribution can be found in many probability books (\cite{wikipedia_normal}). Let \( X \sim \mathcal{N}(\mu, \sigma^2) \), where the PDF of \( X \) is
\[
f_X(x) = \frac{1}{\sqrt{2\pi \sigma^2}} \exp\left(-\frac{(x - \mu)^2}{2\sigma^2} \right)
\]
We compute the first four non-central moments
\[
\mathbb{E}[X^n] = \int_{-\infty}^\infty x^n f_X(x)\, dx
\]
We use the substitution 

\[
z = \frac{x - \mu}{\sigma} \quad \Rightarrow \quad x = \mu + \sigma z,\quad dx = \sigma\, dz
\]
Then

\[
\mathbb{E}[X^n] = \int_{-\infty}^{\infty} (\mu + \sigma z)^n \cdot \frac{1}{\sqrt{2\pi}} e^{-z^2/2} dz
\]
We expand using the binomial theorem

\[
(\mu + \sigma z)^n = \sum_{k=0}^n \binom{n}{k} \mu^{n-k} \sigma^k z^k
\]
Therefore

\[
\mathbb{E}[X^n] = \sum_{k=0}^n \binom{n}{k} \mu^{n-k} \sigma^k \mathbb{E}[Z^k]
\]
where \( Z \sim \mathcal{N}(0, 1) \), and 
\[
\mathbb{E}[Z^n] = \int_{-\infty}^{\infty} z^n \cdot \frac{1}{\sqrt{2\pi}} e^{-z^2/2} dz
\]
Note that
\begin{itemize}
  \item If \( n \) is odd, \( z^n \) is an odd function and the integral over symmetric limits is 0.
  \item For \( n = 2 \), this is the variance of the standard normal: \( \mathbb{E}[Z^2] = 1 \).
  \item For \( n = 4 \), use known identity: \( \mathbb{E}[Z^4] = 3 \).
\end{itemize}
Here we use the identity for even moments of the standard normal:

\begin{equation}\label{gamm}
E[Z^{2n}] = \frac{(2n)!}{2^n n!}
\end{equation}
Let's verify \eqref{gamm}. Since 
$$
\mathbb{E}[Z^{2n}] = \int_{-\infty}^{\infty} z^{2n} \frac{1}{\sqrt{2\pi}} e^{-z^2/2} \, dz
$$
and the integrand is even, this simplifies to
$$
\mathbb{E}[Z^{2n}] = \frac{2}{\sqrt{2\pi}} \int_{0}^{\infty} z^{2n} e^{-z^2/2} \, dz.
$$
Let $ t = z^2/2 $, so $ z = \sqrt{2t} $, $ dz = \frac{\sqrt{2}}{2\sqrt{t}} dt $ and we have

$$
\mathbb{E}[Z^{2n}] = \frac{2}{\sqrt{2\pi}} \int_{0}^{\infty} (2t)^n e^{-t} \frac{\sqrt{2}}{2\sqrt{t}} \, dt = \frac{2^n}{\sqrt{\pi}} \int_{0}^{\infty} t^{n - 1/2} e^{-t} \, dt.
$$
The integral is the Gamma function $ \Gamma(n + 1/2) $. Using the Gamma property

$$
\Gamma\left(n + \frac{1}{2}\right) = \frac{(2n)!}{4^n n!} \sqrt{\pi}
$$
we substitute back
$$
\mathbb{E}[Z^{2n}] = \frac{2^n}{\sqrt{\pi}} \cdot \frac{(2n)!}{4^n n!} \sqrt{\pi} = \frac{(2n)!}{2^n n!}.
$$
Therefore

$$
\mathbb{E}[Z^{2n}] = \frac{(2n)!}{2^n n!}
$$
Now we compute each moment
\begin{align*}
\mathbb{E}[X] &= \mu + \sigma \mathbb{E}[Z] = \mu, \\
\mathbb{E}[X^2] &=  \mu^2 + 2\mu \sigma \mathbb{E}[Z] + \sigma^2 \mathbb{E}[Z^2] = \mu^2 + \sigma^2, \\
\mathbb{E}[X^3] &= \mu^3 + 3\mu^2 \sigma \mathbb{E}[Z] + 3\mu \sigma^2 \mathbb{E}[Z^2] + \sigma^3 \mathbb{E}[Z^3]
= \mu^3 + 3\mu \sigma^2, \\
\mathbb{E}[X^4] &= \mu^4 + 4\mu^3 \sigma \mathbb{E}[Z] + 6\mu^2 \sigma^2 \mathbb{E}[Z^2] + 4\mu \sigma^3 \mathbb{E}[Z^3] + \sigma^4 \mathbb{E}[Z^4]
= \mu^4 + 6\mu^2 \sigma^2 + 3\sigma^4.
\end{align*}

\section{Data Availability} The manuscript has no associated data.

\end{document}